\documentclass[11pt]{article}
\usepackage[latin1]{inputenc}
\usepackage{epsfig}
\usepackage{color}
\usepackage[british,english]{babel}
\usepackage{amsthm}
\usepackage{amsmath}
\usepackage{amsfonts}
\usepackage{amssymb}
\usepackage{graphicx}
\newcommand{\sgn}{\text{sgn}}

\newtheorem{corollary}{Corollary}[section]

\newtheorem{lemma}[corollary]{Lemma}
\newtheorem{proposition}[corollary]{Proposition}
\newtheorem{remark}[corollary]{Remark}
\newtheorem{theorem}[corollary]{Theorem}
\newfont{\sBlackboard}{msbm10 scaled 900}

\newcommand{\mylabel}[1]{\label{#1}
            \ifx\undefined\stillediting
            \else \fbox{$#1$}\fi }
\newcommand{\BE}{\begin{equation}}

\newcommand{\EEQ}{\end{equation}}
\newcommand{\rfb}[1]{\mbox{\rm
   (\ref{#1})}\ifx\undefined\stillediting\else:\fbox{$#1$}\fi}

\newfont{\Blackboard}{msbm10 scaled 1200}

\newfont{\roma}{cmr10 scaled 1200}

\def\CC{\rm \hbox{C\kern-.56em\raise.4ex
         \hbox{$\scriptscriptstyle |$}\kern+0.5 em }}

\def\Frac{\displaystyle\frac}
\def\Int{\displaystyle\int}
\def\n{|\kern -.05cm{|}\kern -.05cm{|}}

\def\R{{\bf \hbox{\sc I\hskip -2pt R}}} 

\newcommand{\mm}    {{\hbox{\hskip 0.5pt}}}

\newcommand{\bluff} {{\hbox{\raise 15pt \hbox{\mm}}}}
\makeatletter
\def\section{\@startsection {section}{1}{\z@}{-3.5ex plus -1ex minus
    -.2ex}{2.3ex plus .2ex}{\large\bf}}
\makeatother
\def\be{\begin{equation}}
\def\ee{\end{equation}}

\date{ }
\begin{document}
\thispagestyle{empty}
\title{\bf Global behavior of the Solutions to
a Class of Nonlinear, Singular  Second Order ODE.}\maketitle

\author{ \center  Mama ABDELLI\\
Université Djillali Liabés, Laboratoire de Mathématique, \\B.P.89. Sidi Bel Abbés 22000,
Algeria.\\
abdelli-mama@yahoo.fr\\}
\medskip\author{ \center  Alain HARAUX (1, 2)\\ 1.  UPMC Univ Paris 06, UMR 7598, Laboratoire Jacques-Louis Lions, \\ F-75005,
Paris, France.\\{2- CNRS, UMR 7598, Laboratoire Jacques-Louis Lions, \\Bo\^{\i}te courrier 187,  75252 Paris Cedex 05,  France.\\
 haraux@ann.jussieu.fr\\}}

\vskip20pt

 \renewcommand{\abstractname} {\bf Abstract}
\begin{abstract} In this paper the initial value problem and global properties of solutions are studied for  the scalar second order ODE: $ (|u'|^{l}u')' +
c|u'|^{\alpha}u' + d|u|^\beta u=0$, where $\alpha,\beta,l,c, d$ are positive constants. In particular, existence, uniqueness and regularity as well as
optimal  decay rates of solutions to 0 are obtained depending on the various parameters, and the oscillatory or non-oscillatory behavior is elucidated .
\end{abstract}
\bigskip\noindent

 {\small \bf AMS classification numbers:} 34A34, 34C10, 34D05, 34E99

\newpage
\section {Introduction}

 In this paper we consider the scalar second order ODE
 \begin{equation}\label{1}
 \Big(|u'|^{l}u'\Big)' + c|u'|^{\alpha}u' + d|u|^\beta u=0,
 \end{equation}
 where $\alpha,\beta, c, d$ are positive constants and $l\geq
 0$.\\ In the special case $l=0$ and $d=1$ we find the simpler equation
\begin{equation}\label{13}
u'' + c|u'|^{\alpha}u' + |u|^\beta u=0.
 \end{equation}
 The solutions of \eqref{13} are global for $t\ge 0$ and both $u$ and $u'$ decay to $0$ as  $t \rightarrow \infty$. This equation was studied in {\bf\cite{har2}} by the second author who used some modified energy function to estimate the rate of decay. In addition, he showed that if $\alpha > \frac{\beta}{\beta+2}$ all non-trivial solutions are oscillatory and
if $\alpha <\frac{\beta}{\beta+2}$ they are non-oscillatory.\\

 The consideration of the more complicated problem \eqref{1} is partially motivated by {\bf\cite{BM}} and {\bf\cite{AB}}  in which a similar but harder (infinite dimensional) problem with nonlinear dissipation $\sigma(t)g(u')$ was studied with application to some PDE in a bounded domain. Under  Neumann or Dirichlet boundary  conditions, and for nonlinearities asymptotically homogeneous near $0$ similar to the ones appearing in \eqref{1}, they proved a sharp decay property of the energy without establishing  a well-posedness result. This followed a  previous work {\bf\cite{AM}} by  Benaissa and Amroun who constructed exact solutions of a quasilinear wave equation of Kirchhoff type with nonlinear source term without dissipative term for some  initial data and showed  finite time blowing up results for some other initial data.\\
 
In this paper we use some techniques from {\bf\cite{har2}} to  estimate the energy decay of the solutions to (\ref{1}) and to  show that all non-trivial solutions of
  (\ref{1}) are oscillatory for $\alpha>
  \frac{\beta(l+1)+l}{\beta+2}$ and $\alpha>l$, non-oscillatory for  $\alpha<
  \frac{\beta(l+1)+l}{\beta+2}$ and $\beta>l$. One major difference with {\bf\cite{har2}} is that here we have to establish well-posedness in a regularity class compatible with the presence of the singular leading term $ \Big(|u'|^{l}u'\Big)' $ in the equation. This singularity prevents $u$ to have a second derivative at all points where $u'$ vanishes for any non-trivial solution. We show that such points are isolated, which allow us to generalize the methods of {\bf\cite{har2}}, sometimes by using the density of regular points after proving identities or inequalities outside the singularities. \\

The paper is organized as follows. In section 2 we prove the existence and  uniqueness of a global  solution $u\in {\mathcal C}^1(\R^+)$ with $|u'|^lu' \in {\mathcal C}^1(\R^+)$ for any initial data in $\R^2$ as well as the isolated character of singular points. In section 3 we prove the energy estimates, then oscillatory and non-oscillatory behavior are delimited at section 4 according to the relations between the various parameters in \eqref{1}. In section 5 we study very precisely the non-oscillatory case in which 2 different decay rates can occur. Finally, section 6 is devoted  to some optimality results. For related results concerning convergence to equilibrium and multiple rates of convergence for second order evolution equations, cf. also {\bf\cite{har1, har3, har4, har5}}.
  \section{The initial value problem for Equation (\ref{1})}
  In this section, we consider the existence and uniqueness of solutions to the initial value problem associated to \eqref{1}. We start by an existence result which does not require any additional restrictions on the parameters $\alpha,\beta, c, d, l$.
\begin{proposition}\label{Pr1}
Let $(u_0,u_1)\in \R^2$.
The problem (\ref{1}) has a  global solution satisfying
$$u\in {\mathcal C}^1(\R^+),\,\,\,\,\,\,|u'|^lu'\in {\mathcal C}^1(\R^+)\,\,\,\,\
\mbox{and}\,\,\,\, u(0)= u_0,\,\,\ u'(0)= u_1.$$
\end{proposition}
\begin{proof}
To show the existence of the solution for (\ref{1}), we consider
{\small\begin{equation}\label{H1}\left\{
\begin{array}{ll}
(\varepsilon +(l+1)|u_\varepsilon'|^l)u_\varepsilon''
+c|u'_\varepsilon|^\alpha u'_\varepsilon +d|u_\varepsilon|^\beta u_\varepsilon=0\\
u_\varepsilon(0)=u_0,\,\,\,\, u'_\varepsilon(0)=u_1.
\end{array} \right.\end{equation}} Here, $\varepsilon > 0$ is a small parameter, devoted to tend to zero.\\

\begin{itemize}
 \item[i)] A priori estimates:\\\\
From (\ref{H1}), we get
{\small\begin{equation*}\label{H2}\left\{
\begin{array}{ll}
u_\varepsilon''
+\Frac{c|u'_\varepsilon|^\alpha u'_\varepsilon +d|u_\varepsilon|^\beta u_\varepsilon}{\varepsilon +(l+1)|u_\varepsilon'|^l}=0\\
u_\varepsilon(0) =u_0,\,\,\,\, u_\varepsilon(0) =u_1.
\end{array} \right.\end{equation*}}
This is the initial value problem for an
 ordinary differential equation, which admits a unique  local solution $u_\varepsilon \in {\mathcal C}^2([0,T_{\max}))$. Given $t \in [0,T_{\max})$ we have the following energy identity:
{\small\begin{equation*}
 \begin{split}
 \frac{d}{dt}\Big[\frac{\varepsilon}{2}|u'_\varepsilon(t)|^2
 +\frac{l+1}{l+2}|u_\varepsilon'(t)|^{l+2}+\frac{d}{\beta+2}|u_\varepsilon(t)|^{\beta+2}\Big] +
 c|u'_\varepsilon(t)|^{\alpha+2} = 0.
 \end{split}
\end{equation*}}
By integrating  over $(0,t)$, we get
{\small\begin{equation*}
 \begin{split}
 &\frac{\varepsilon}{2}|u'_\varepsilon(t)|^2
 +\frac{l+1}{l+2}|u_\varepsilon'(t)|^{l+2}+\frac{d}{\beta+2}|u_\varepsilon(t)|^{\beta+2} +
 c\Int_0^t|u'_\varepsilon(s)|^{\alpha+2}\,ds \\&= \frac{\varepsilon}{2}|u_1|^2
 +\frac{l+1}{l+2}|u_1|^{l+2}+\frac{d}{\beta+2}|u_0|^{\beta+2}.
 \end{split}
\end{equation*}}
Hence, for some constants $M_1, M_2$ independent of $\varepsilon$ we have 
\begin{equation}\label{H4}
\forall t \in [0,T_{\max}), \quad |u_\varepsilon(t)|\leq M_1,\,\,\,|u'_\varepsilon(t)|\leq M_2.
\end{equation}
In particular $T_{\max}=+\infty$,\,\, $u_\varepsilon$ is global, $u_\varepsilon \in
{\mathcal C}^2(\R^+)$ and $u_\varepsilon,\,\,\,u'_\varepsilon$ are uniformly bounded. Now we have
{\small\begin{equation*}
 \begin{split}
\Big|\Big(|u_\varepsilon'(t)|^lu'_\varepsilon(t)\Big)'\Big| &=(l+1)|u'(t)|^l|u''_\varepsilon(t)|
\\& \leq \Big|(\varepsilon + (l+1)|u'_\varepsilon(t)|^l)u''_\varepsilon(t)\Big|,
 \end{split}
\end{equation*}}
and by using (\ref{H1}) and (\ref{H4}), we deduce
\begin{equation}\label{H5}
\Big|\Big(|u_\varepsilon'(t)|^lu'_\varepsilon(t)\Big)'\Big| \leq M_4.
\end{equation}
Therefore the function $w_\varepsilon(t):= |u'_\varepsilon(t)|^lu'_\varepsilon(t)$ is uniformly Lipschitz continuous on $\R^+$ independently of $\varepsilon$.
Then the family of functions  $u'_\varepsilon(t)=|w_\varepsilon(t)|^{\frac{1}{l+1}} \sgn w_\varepsilon(t)$ is uniformly equicontinous (actually Holder continuous ) on $(0,T)$.\\
 \item[ii) ]
 Passage to the limit:\\\\
As a consequence of Ascoli's theorem and a priori estimate
(\ref{H4}), we may extract a subsequence which is still denoted for simplicity by
$(u_\varepsilon)$ such that for every $T> 0$
$$
u_\varepsilon \rightarrow u\,\,\,\,\,\mbox{in}\,\,\, {\mathcal C}^1(0,T)
$$
as $\varepsilon$ tends to $0$. Integrating (\ref{H1}) over $(0,t)$, we get
{\small\begin{equation}\label{li}
 \begin{split}
 |u'_\varepsilon(t)|^lu'_\varepsilon(t)-|u'_\varepsilon(0)|^lu'_\varepsilon(0) &=
 -c\Int_0^t|u'_\varepsilon(s)|^\alpha u'_\varepsilon(s)\,ds-d\Int_0^t|u_\varepsilon(s)|^\beta u_\varepsilon(s)\,ds
 -\varepsilon \Int_0^tu''_\varepsilon(s)\,ds\\&=
 -c\Int_0^t|u'_\varepsilon(s)|^\alpha u'_\varepsilon(s)\,ds-d\Int_0^t|u_\varepsilon(s)|^\beta u_\varepsilon(s)\,ds
 -\varepsilon (u'_\varepsilon(t)- u_1).
 \end{split}
\end{equation}}
From (\ref{li}), we then have, as $\varepsilon$ tends to $0$
$$
|u'_\varepsilon|^lu'_\varepsilon \rightarrow -c\Int_0^t|u'(s)|^\alpha u'(s)\,ds-d\Int_0^t|u(s)|^\beta u(s)\,ds
+|u'(0)|^lu'(0)
\,\,\,\,\,\mbox{in}\,\,\, {\mathcal C}^0(0,T).
$$
Hence \begin{equation}\label{ABS}
|u'|^lu'= -c\Int_0^t|u'(s)|^\alpha u'(s)\,ds-d\Int_0^t|u(s)|^\beta u(s)\,ds+|u'(0)|^lu'(0),
\end{equation}
and $|u'|^lu' \in {\mathcal C}^1(0,T)$. finally by differentiating (\ref{ABS}) we conclude that $u$ is a solution of (\ref{1}).
 \end{itemize}
  This completes the proof of Proposition \ref{Pr1}.\end{proof}
\begin{remark}
 The uniqueness of solutions of (\ref{1}) with $(u_0,u_1)$  given will be proved under some restrictions on the parameters $\alpha,\beta, l$. The next proposition concerns the uniqueness result for $(u_0,u_1) =(0,0)$.
 \end{remark}
\begin{proposition}\label{Pro}
Assume that  $l\leq \inf (\alpha,\beta)$. Then for any interval $J$ and any  $\tau\in J$ , if a solution $u$ of (\ref{1}) satisfies
$$u\in {\mathcal C}^1(J),\,\,\,\,\,\,\,|u'|^lu'\in {\mathcal C}^1(J)\,\,\,\,\, \mbox{and}\,\,\,\
u(\tau)=u'(\tau)=0,$$
then  $u\equiv 0$.
\end{proposition}
\begin{proof} Since (\ref{1}) is autonomous, by replacing $u(t)$ by $u(t+\tau)$ and J by $J-\tau$ we are reduced to the case $\tau = 0$.  Then starting with the case $\min J= 0$, from (\ref{1}), we have
$$
(|u'|^lu')'=-c|u'|^\alpha u' -d|u|^\beta u.
$$
Integrating over $(0,t) \in J$ and using  $u'(0)=0$, we have
$$
|u'(t)|^lu'(t)=-c\Int_0^t|u'(s)|^\alpha u'(s)\,ds -d\Int_0^t|u(s)|^\beta u(s)\,ds.
$$
Since $u(0)=0$, we deduce
\begin{equation*}
\begin{split}
\Big||u'(t)|^lu'(t)\Big|&\leq c\Big|\Int_0^t|u'(s)|^\alpha u'(s)\,ds\Big|+
d\Big|\Int_0^t|u(s)|^\beta u(s)\,ds\Big| \\&\leq
c\Big|\Int_0^t|u'(s)|^\alpha u'(s)\,ds\Big|+
d\Int_0^t\Big(\Int_0^s|u'(\tau)|\,d\tau\Big)^{\beta+1}
\end{split}
\end{equation*}
and this implies
\begin{equation}\label{VM1}
\Big||u'(t)|^lu'(t)\leq
c\Big|\Int_0^t|u'(s)|^\alpha u'(s)\,ds\Big|+t d
\Big(\Int_0^t|u'(s)|\,ds\Big)^{\beta+1}.
\end{equation}
Applying H\"older inequality, we have
\begin{equation}\label{VM}
\Int_0^t|u'(s)|\,ds\leq \Big(\Int_0^t|u'(s)|^{l+1}\,ds\Big)^{\frac{1}{l+1}}t^{\frac{l}{l+1}}.
\end{equation}
Hence, by (\ref{VM1}) and (\ref{VM}), we have
$$
\Big||u'(t)|^lu'(t)\Big|\leq c\Int_0^t\Big||u'(s)|^\alpha u'(s)\Big|\,ds+C(T,d)
\Big(\Int_0^t\Big||u'(s)|^{l}u'(s)\Big|\,ds\Big)^{\frac{\beta+1}{l+1}}.
$$
Since $\alpha\geq l$ and $\beta\geq l$,  and since $ |u'(t)|$ is bounded we obtain
$$ \Big||u'(t)|^lu'(t)\Big|\leq C(T)\Int_0^t\Big||u'(s)|^{l}u'(s)\Big|\,ds. $$
By introducing  $w(t)= \Big||u'(t)|^lu'(t)\Big|$, we see that
$$
w(t) \leq C(T) \Int_0^t w(s)\,ds,\,\,\,\,\,\mbox{for}\,\,\,\,t\in [0,T]\subset J.
$$
which, by Gronwall's inequality, implies $w\equiv 0$ on $[0,T]$. Since $T$ is arbitrarily we conclude
 $w\equiv 0$ on $J$. A similar proof is valid if $0=\max J$, and finally for any $J$ with $0\in J$.
This concludes the proof of Proposition \ref{Pro} \end{proof}
\begin{remark}
\rm{The result of proposition \ref{Pro}  has very important consequences which will be useful throughout the rest of the paper.}
\end{remark}
\begin{corollary}\label{finite}
Let $u \in {\mathcal C}^1(J)$ be any solution of (\ref{1}) with $|u'|^lu'\in {\mathcal C}^1(J)$,\,\,\
$u \not\equiv 0$. Then for each compact interval $K\subset J$ the set
$F=\{t\in K,\,\, u(t)=0\}$ is finite.
\end{corollary}
\begin{proof} By contradiction. Assuming that $F$ is infinite, let $\tau\in K$ be an accumulation point of $F$.
Let $t_n \in F$ with $t_n \rightarrow \tau$ as $n \rightarrow \infty$. Clearly $u(\tau)= 0$ and since between
 $t_n$ and $t_{n+1}$ there is $A_n$ with $u'(A_n)=0$ (Rolle's Theorem)  we also have $u'(\tau)=0$.
  By Proposition \ref{Pro} we conclude that
 $u\equiv 0$ on $J$.
 \begin{corollary}\label{f^2}
Let $u \in {\mathcal C}^1(J)$ be any solution of (\ref{1}) with $|u'|^lu'\in {\mathcal C}^1(J)$,\,\,\
$u \not\equiv 0$. Then for each compact interval $K\subset J$ the set
$G=\{t\in K,\,\, u'(t)=0\}$ is finite.
\end{corollary}
{\bf{Proof.}}  Suppose that
$$
u'(t_n)=0\,\,\,\,\mbox{when}\,\,\,\ n\rightarrow \infty \Rightarrow |u'(t_n)|^lu'(t_n)= 0.
$$
There exist $A_n\in (t_n,t_{n+1})$ such that $(|u'|^lu')'(A_n)= 0$
 when $t_n\rightarrow \tau$ and $A_n\rightarrow \tau$
$$
(|u'|^lu')'(\tau) + c|u'(\tau)|^\alpha u'(\tau)+ d|u(\tau)|^\beta u(\tau)=0,
$$
since $(|u'|^lu')'(\tau) + c|u'(\tau)|^\alpha u'(\tau)=0$, therefore
 $$
 |u(\tau)|^\beta u(\tau)=0\Rightarrow u(\tau)=0
 $$
 $$ u(\tau)=u'(\tau)=0\Rightarrow u\equiv 0. $$\end{proof}
\begin{proposition}\label{Propu}
Let $a \not= 0$.
Then for $J$ an interval containing $0$ and such that $\vert J\vert$ is small enough, equation \eqref{1} has at most one  solution satisfying
$$u\in {\mathcal C}^1(J),\,\,\,\,\,\,|u'|^lu'\in {\mathcal C}^1(J)\,\,\,\,\
\mbox{and}\,\,\,\, u_0=a,\,\,\ u_1=0.$$
\end{proposition}
\begin{proof}
 By the odd character of the equation it is sufficient to consider the case $a>0$.  From (\ref{1}), we obtain $(|u'|^lu')'(0)=-a^{\beta+1}$ and $|u'|^l\geq \eta t^{\frac{l}{l+1}}$
  for $\vert t\vert$ small enough and some $\eta>0$.\\
  
 Integrating (\ref{1}) over $(0,t)$, we have since $u'(0) = 0$
 $$
 |u'(t)|^lu'(t) = -d\Int_0^t|u(\tau)|^\beta u(\tau)\,d\tau -c\Int_0^t|u'(\tau)|^\alpha u'(\tau)\,d\tau.
 $$
 Let $u(t)$ and $v(t)$ be two solutions, then $w(t)=u(t)-v(t)$ satisfies
 {\small\begin{equation}\label{hh}
\begin{split}
 |u'(t)|^lu'(t)- |v'(t)|^lv'(t)&= -d\Int_0^t(|u(\tau)|^\beta u(\tau)-|v(\tau)|^\beta v(\tau))\,d\tau\\& -c\Int_0^t|(u'(\tau)|^\alpha u'(\tau)-|v'(\tau)|^\alpha v'(\tau))\,d\tau
 \end{split}
\end{equation}}
Furthermore,
$$
||u'(t)|^lu'(t)- |v'(t)|^lv'(t)|\geq \inf \{|u'|^l,|v'|^l\}|u'-v'|\geq \eta t^{\frac{l}{l+1}}|u'-v'|,
$$
and from (\ref{hh}), we now deduce
{\small\begin{equation}\label{hhh}
\begin{split}
 |w'(t)|&\leq\frac{C}{t^{\frac{l}{l+1}}}\Int_0^t\Int_0^t |w'(\tau)|\,d\tau\,ds+
 \frac{C}{t^{\frac{l}{l+1}}}\Int_0^t|w'(\tau)|\,d\tau\\&\leq C(T)t^{-\frac{l}{l+1}}\Int_0^t |w'(\tau)|\,d\tau.
 \end{split}
\end{equation}}
Setting  $\phi(t)=\Int_0^t|w'(\tau)|\,d\tau$, by solving \eqref{hhh} on $[\delta, t]$ we obtain 
$$
\phi(t)\leq \phi(\delta)e^{C(T)\int_\delta^t s^{-l/(l+1)}\,ds}
$$
And by letting $\delta \rightarrow 0$ we conclude that  $w(t)=0.$ A similar argument gives the uniqueness for $t$ negative with $\vert t\vert$ small enough.
\end{proof}

We now state the main result of this section

\begin{theorem}\label{Th} Assume that  $l\leq \inf (\alpha,\beta)$. then for any $(u_0,u_1)\in \R^2$,
the equation (\ref{1}) has a  unique global solution satisfying
$$u\in {\mathcal C}^1(\R^+),\,\,\,\,\,\,|u'|^lu'\in {\mathcal C}^1(\R^+)\,\,\,\,\
\mbox{and}\,\,\,\, u_0=u(0),\,\,\ u_1=u'(0).$$
\end{theorem}\begin{proof}The existence follows from Proposition \ref{Pr1}. Uniqueness follows from Proposition \ref{Pro} and Proposition \ref{Propu} since in the non-singular case $u'(t_0) \not= 0$ local uniqueness is a consequence of classical results, while in the case $u'(t_0) = 0$, we can either use Proposition \ref{Pro}  if $u(t_0) = 0$ or Proposition \ref{Propu} if $u(t_0) \not = 0$.
 \end{proof}

 \begin{remark}
 If $u \not\equiv 0$, at any point $t_0$ where $u(t_0) \neq 0$ and $u'(t_0)=0$, the second derivative $u''(t_0)$
  does not exist. At least when $\alpha > l-1$. Indeed for $t\neq t_0$  and $|t-t_0|< \varepsilon$ we have
  $$
  u''(t)=-\frac{c}{l+1}|u'|^{\alpha-l}u'(t)-\frac{d|u|^\beta u}{(l+1)|u'(t)|^l},
  $$
  hence, as $t\rightarrow t_0$,\,\, $u''(t)$ has a constant sign and
  $$
  \vert u''(t)\vert \rightarrow +\infty\,\,\,\,\,\mbox{as}\,\,\, t\rightarrow t_0.
  $$
 \end{remark}
 This implies that $u'$ is not differentiable at $t_0$.
 \section {Energy estimates  for equation (\ref{1})}
  We define the energy associated to the solution of the problem by
the following formula
\begin{equation}\label{2}
E(t) = \Frac{l+1}{l +2}|u'|^{l +2} + \Frac{d}{\beta +2}|u|^{\beta +2}.
\end{equation}
By multiplying equation (\ref{1}) by $u'$, we obtain that on any interval where $u$ is
${\mathcal C}^2$,\,\, $E(t)$ is ${\mathcal C}^1$ with
\begin{equation}\label{3}
\Frac{d}{dt}E(t) = -c|u'|^{\alpha+2}\leq 0.
\end{equation}
In particular (\ref{3}) holds, whenever $u'(t)\neq 0$.\\ Now let $t_0$ be such that $u'(t_0)=0$. As a consequence of Corollary  \ref{f^2} there exists
$\varepsilon> 0$ such that $u\in {\mathcal C}^2((t_0,t_0+\varepsilon]\cup[t_0-\varepsilon, t_0))$.
Integrating (\ref{3}) over $(\tau,t)$, 
$$
E(t)-E(\tau)=-c\Int_\tau^t|u'(s)|^{\alpha+2}\,ds,\,\,\,\,\, t_0< \tau \leq t \leq t_0+\varepsilon,
$$
By letting $\tau\rightarrow t_0$, we obtain
$$
E(t)-E(t_0)=-c\Int_{t_0}^t|u'(s)|^{\alpha+2}\,ds,
$$
In particular we obtain that $E$ is right-differentiable at $t_0$ with right-derivative equal to $-c|u'|^{\alpha+2} $. A similar calculation on the left allows to conclude that $E$ is differentiable at  $t_0$ and finally  (\ref{3}) is true at any point.

\begin{theorem}\label{Th2}
Assuming $\alpha> l$, there exists a positive constant
$\eta$ such that if $u$ is any solution of (\ref{1}) with $E(0)\neq 0$
\begin{equation}\label{22}
\liminf_{t\rightarrow +\infty}t^{\frac{l+2}{\alpha-l}}E(t)\geq \eta.
\end{equation}
Moreover,
\begin{itemize}
 \item[(i)] if $\alpha \geq \frac{\beta(1+l)+l}{\beta+2}$,
     then there is a constant $C(E(0))$ depending on
     $E(0)$ such that
     $$
     \forall t\geq 1,\,\,\,\ E(t)\leq C(E(0)) t^{-\frac{l+2}{\alpha-l}},
     $$
     \item[(ii)] if $\alpha < \frac{\beta(1+l)+l}{\beta+2}$ ,
         then there is a constant $C(E(0))$ depending
         on $E(0)$ such that
     $$
     \forall t\geq 1,\,\,\,\ E(t)\leq C(E(0)) t^{-\frac{(\alpha+1)(\beta+2)}
     {\beta-\alpha}}.
     $$
 \end{itemize}
\end{theorem}

\begin{proof}
From (\ref{2}), we have
$$
|u'(t)|^{\alpha+2}\leq C(l,\alpha) E(t)^{\frac{\alpha+2}{l+2}},
$$
where $C(l,\alpha)$ is a positive constant, hence from (\ref{3}) we
deduce
$$
\frac{d}{dt}E(t)\geq -C(l,\alpha) E(t)^{\frac{\alpha+2}{l+2}},
$$
Assuming  $\alpha>l$  we derive
\begin{equation*}
\begin{split}
\frac{d}{dt}E(t)^{-\frac{\alpha-l}{l+2}}&=
-\frac{\alpha-l}{l+2}E(t)^{-\frac{\alpha+2}{l+2}}E'(t)\\&\leq
\frac{\alpha-l}{l+2} c C(l,\alpha) = C_1.
\end{split}
\end{equation*}
By integrating, we get
$$
E(t)\leq \Big(E(0)^{-\frac{\alpha-l}{l+2}}+C_{1}t\Big)^{-\frac{l+2}{\alpha-l}},
$$
implying
$$
\liminf_{t\rightarrow +\infty}t^{\frac{l+2}{\alpha-l}}E(t)\geq \eta =
C_1^{-\frac{l+2}{\alpha-l}}.
$$
Hence (\ref{22}) is proved. Now, in order to show $(i)$ and $(ii)$,
we consider the perturbed energy function
\begin{equation}\label{4}
E_\varepsilon (t) = E(t) + \varepsilon  |u|^{\gamma}u|u'|^l u',
\end{equation}
where $l > 0$,\,\,$\gamma >0$ and $\varepsilon > 0.$\\
By Young's inequality, we have
$$
\vert |u|^{\gamma}u|u'|^l u'\vert \leq
 c_1|u|^{(\gamma+1)(l+2)}+c_2|u'|^{l+2},
$$
we choose $\gamma$ so that $(\gamma+1)(l+2)\geq \beta +2$ which reduces  to
\begin{equation}\label{55}
\gamma \geq \Frac{\beta -l}{l +2}.
\end{equation}
Hence, since $u$ is bounded, along the trajectory we have for some   $C_1> 0, M>0$
 {\small\begin{equation}
\begin{split}\label{5}
\vert |u|^{\gamma}u|u'|^l u'\vert &\leq C_1 |u|^{\beta+2}+c_2
|u'|^{l+2}\\&\leq M E(t).
\end{split}
\end{equation}}
Then, by using (\ref{5}), we obtain from (\ref{4})
$$
(1-M\varepsilon)E(t)\leq E_\varepsilon(t) \leq (1+M\varepsilon)E(t).
$$
Taking  $\varepsilon\leq \Frac{1}{2M}$, we deduce
\begin{equation}\label{6}
\forall t\geq 0,\,\,\, \Frac{1}{2}E(t)\leq E_\varepsilon(t)\leq 2 E(t).
\end{equation}
On the other hand
{\small\begin{equation}\label{33a}
\begin{split}
E'_\varepsilon(t) &= -c|u'|^{\alpha+2} + \varepsilon
(\gamma+1)|u|^\gamma|u'|^{l+2}+ \varepsilon
|u|^{\gamma}u(-c|u'|^{\alpha}u'-d|u|^{\beta}u),\\ E'_\varepsilon(t)
&= -c|u'|^{\alpha+2} -\varepsilon d |u|^{\gamma+\beta+2} +
\varepsilon (\gamma+1)|u|^\gamma|u'|^{l+2} -\varepsilon
c|u|^\gamma u |u'|^\alpha u'.
\end{split}
\end{equation}}
By using Young's inequality, with the conjugate exponents
$\frac{\alpha+2}{l+2}$ and $\frac{\alpha+2}{\alpha-l}$, we get
$$
|u|^\gamma|u'|^{l+2} \leq \delta|u|^{\frac{\gamma(\alpha+2)}{\alpha-l}}
+ C(\delta)|u'|^{\alpha+2},
$$
we assume $$\frac{(\alpha+2)\gamma}{\alpha-l}\geq \beta+\gamma
+2,$$ which reduces to the condition
\begin{equation}\label{555}
 \gamma \geq
\frac{(\beta+2)(\alpha-l)}{l+2},
\end{equation}
and taking  $\delta$ small enough, we have for some $P>0$
\begin{equation}\label{123}
\varepsilon (\gamma+1)|u|^\gamma|u'|^{l+2} \leq
\frac{\varepsilon d}{4}|u|^{\beta+\gamma +2}
+ P\varepsilon|u'|^{\alpha+2}.
\end{equation}
Using (\ref{123}), we have from (\ref{33a}) that
\begin{equation}\label{234}
E_\varepsilon'(t)\leq (-c+P\varepsilon)|u'|^{\alpha+2}
-\Frac{3d\varepsilon}{4}|u|^{\beta+\gamma+2}
-c\varepsilon|u|^{\gamma}u|u'|^{\alpha}u'.
\end{equation}
Applying Young's inequality, with the conjugate exponents
$\frac{\alpha+2}{\alpha+1}$ and $\frac{1}{\alpha+2}$, we have
$$
-c|u|^{\gamma}u|u'|^{\alpha}u'\leq \delta|u|^{(\gamma+1)(\alpha+2)}
+C'(\delta)|u'|^{\alpha+2}.
$$
This term will be dominated by the negative terms assuming
$$(\gamma+1)(\alpha+2)\geq \beta+\gamma+2\Leftrightarrow
(\alpha+1)(\gamma+1)\geq \beta+1.$$ This is equivalent to the condition
\begin{equation}\label{5555}
\gamma \geq \frac{\beta-\alpha}{\alpha+1},
\end{equation}
and taking $\delta$ small enough
$$
-c\varepsilon|u|^{\gamma}u|u'|^{\alpha}u'\leq \Frac{\varepsilon d}{4}|u|^{\beta+\gamma+2}
+P'\varepsilon|u'|^{\alpha+2}.
$$
By replacing in (\ref{234}), we have
$$
E'_\varepsilon(t)\leq(-c+Q\varepsilon)|u'|^{\alpha+2}-\Frac{\varepsilon d}{2}|u|^{\beta+\gamma+2},
$$
where $Q=P+P'$. By choosing $\varepsilon$ small, we get
{\small\begin{equation}\label{0}
 \begin{split}
E'_\varepsilon(t)&\leq-\frac{\varepsilon}{2}\Big(|u'|^{\alpha+2}
+|u|^{\beta+\gamma+2}\Big)\\&\leq -\frac{\varepsilon}{2}
\Big((|u'|^{l+2})^{\frac{\alpha+2}{l+2}}
+(|u|^{\beta+2})^{\frac{\beta+\gamma+2}{\beta+2}}\Big).
\end{split}
\end{equation}}
This inequality will be satisfied under the assumptions (\ref{55}), (\ref{555}) and (\ref{5555}) which lead to the sufficient condition
\begin{equation}
\gamma \geq \gamma_0 = \max\Big\{\frac{\beta-l}{l+2},\frac{(\beta+2)(\alpha-l)}{l+2},
\frac{\beta-\alpha}{\alpha+1}\Big\}.
\end{equation}
We now  distinguish $2$ cases.

\begin{itemize}

 \item[(i)] If $\alpha \geq \Frac{\beta(1+l)+l}{\beta +2}$,
     then clearly $\Frac{(\beta+2)(\alpha-l)}{l+2} \geq
     \Frac{\beta-l}{l+2}$
moreover
$$
\Frac{\beta-\alpha}{\alpha+1}=\Frac{\beta+1}{\alpha+1}-1 \leq
\Frac{\beta+1}{\frac{\beta(1+l)+l}{\beta+2}+1}-1 = \Frac{\beta-l}{l+2}.
$$
In this case $\gamma_0=\frac{(\beta+2)(\alpha-l)}{l+2}$ and
choosing $\gamma=\gamma_0$, we find
$$
\beta+\gamma+2 = \Frac{\alpha+2}{l+2}(\beta+2),
$$
since $ \Frac{\beta+\gamma+2}{\beta+2} =1+
\Frac{\alpha-l}{l+2}$, replacing in (\ref{0}), we obtain for some $\rho>0$
\begin{equation}
\begin{split}
E'_\varepsilon(t)\leq
 -\rho E(t)^{1+\frac{\alpha-l}{l+2}}\leq -\rho'
E_\varepsilon(t)^{1+\frac{\alpha-l}{l+2}},
\end{split}
\end{equation}
where $\rho$ and $\rho'$ are a positive constants.

 \item[(ii)] If $ \alpha < \Frac{\beta(l+1)+l}{\beta+2}$ then
     $\Frac{(\beta+2)(\alpha-l)}{l+2}<\Frac{\beta-l}{l+2}$
     and
$$
\Frac{\beta-\alpha}{\alpha+1}-\Frac{\beta-l}{l+2}
= \Frac{(\beta-\alpha)(l+2)-(\beta-l)(\alpha+1)}{(\alpha+1)(l+2)}
=\Frac{\beta(l+1)+l-\alpha(\beta+2)}{(\alpha+1)(l+2)}> 0.
$$
In this case $\gamma_0= \Frac{\beta-\alpha}{\alpha+1}$ and
choosing $\gamma= \gamma_0,$ we find
\begin{equation}\label{567}
\beta+\gamma+2 =(\beta+2)\Big(1+\Frac{\gamma}{\beta+2}\Big)
=(\beta+2)\Big(1+\Frac{\beta-\alpha}{(\alpha+1)(\beta+2)}\Big),
\end{equation}
since $\gamma > \Frac{(\beta+2)(\alpha-l)}{l+2}$, we have
$$
\Frac{\beta+\gamma+2}{\beta+2}=1 + \Frac{\gamma}{\beta+2}> 1+\Frac{\alpha-l}{l+2}
=\Frac{\alpha+2}{l+2},
$$
replacing in (\ref{0}), we obtain {\small\begin{equation*}
 \begin{split}
E'_\varepsilon(t)&\leq -\frac{\varepsilon}{2}
(|u'|^{\frac{(\alpha+2)(l+2)}{l+2}}
+|u|^{\frac{(\beta+\gamma+2)(\beta+2)}{\beta+2}})\\&\leq
-\delta\varepsilon (|u'|^{(l+2)}
+|u|^{(\beta+2)})^\frac{\beta+\gamma+2}{\beta+2},
\end{split}
\end{equation*}}
for some $\delta> 0$. Using (\ref{567}), we have
\begin{equation}
E'_\varepsilon(t)\leq -\rho E^{(1+\frac{\beta-\alpha}{(\alpha+1)(\beta+2)})}
\leq-\rho' E_\varepsilon(t)^{(1+\frac{\beta-\alpha}{(\alpha+1)(\beta+2)})},
\end{equation}
where $\rho$ and $\rho'$ are positive constants.

 \end{itemize}

This completes the proof of Theorem \ref{Th2}.
\end{proof}

\begin{corollary}\label{CO}
If $\alpha \geq \frac{\beta(l+1)+l}{\beta+2}$, then there is a
constant $C$ depending on $E(0)$ such that
$$
\forall t\geq 1,\,\,\,\, |u(t)|\leq C t^{-\frac{l+2}{(\alpha-l)(\beta+2)}},
$$
$$
\forall t\geq 1,\,\,\,\, |u'(t)|\leq C t^{-\frac{1}{\alpha-l}}.
$$
\end{corollary}
\section{Oscillatory behavior of solutions} Throughout this section, we assume $ l\le \beta $ and $ l<\alpha $ in order for Corollary  \ref{finite} to be applicable and Theorem \ref{Th2} to make sense. We start with a result showing that for $\alpha$ large, all non-trivial solutions are oscillatory.  More precisely we have
\begin{theorem}\label{Th1}
Assume that
\begin{equation}\label{CO1}
\alpha > \Frac{\beta(l+1)+l}{\beta+2}
\end{equation}
or
\begin{equation}\label{CO2}
\alpha = \Frac{\beta(l+1)+l}{\beta+2}\mbox{and}\,\,\,
c< c_0 = (\beta+2)\Big(\Frac{(\beta+2)(l+1)}
{d(\beta+1)(l+2)}\Big)^{\frac{\beta+1}{\beta+2}}.
\end{equation}
Then, any solution $u(t)$ of (\ref{1}) which is
not identically $0$ changes sign on each interval
$(T,\infty)$ and the same thing is true for $u'(t)$  .
\end{theorem}
\begin{proof} We proceed in 2 steps. \\

Step 1. For any $T>0$,  $u'(t)$ has at least a zero on $[T, \infty)$. Indeed assuming that $u'(t)$ has a constant sign on $[T, \infty)$, then  $u(t)$ has also a constant sign, opposite to that of $u'$.

 We define a polar coordinate system  by
\begin{equation}\label{c1}
\Big(\Frac{d(l+2)}{(\beta+2)(l+1)}\Big)^{\frac{1}{2}}|u|^{\frac{\beta}{2}}u =
r(t)\cos \theta(t),\,\,\,\,\,\,\,\,\, |u'|^{\frac{l}{2}}u' = r(t)\sin \theta(t),
\end{equation}
where $r$ and $\theta$ are two $C^1$ functions and
$r(t)=\Big(\Frac{l+2}{l+1}E(t)\Big)^{\frac{1}{2}}> 0$, we
have
$$
r'=-c\Frac{l+2}{2(l+1)}r^{\frac{2(\alpha+2)}{l+2}-1}
|\sin \theta|^{\frac{2(\alpha+2)}{l+2}}.
$$
From (\ref{c1}), we have
\begin{equation*}\label{c3}
|u'|^l u' =r^{\frac{2(l+1)}{l+2}}\sin \theta|\sin \theta|^{\frac{l}{l+2}},
\end{equation*}
and
{\small\begin{equation*}\label{c2}
 \begin{split}
\Big(|u'|^{l}u'\Big)'= -c
r^{\frac{2(\alpha+1)}{l+2}}\sin
\theta|\sin\theta|^{\frac{2(\alpha+2)}{l+2}+\frac{l}{l+2}}+
\frac{2(l+1)}{l+2}r^{\frac{2(l+1)}{l+2}}\theta' \cos \theta|\sin
\theta|^{\frac{l}{l+2}}.
\end{split}
\end{equation*}}
From (\ref{c1}), we have
\begin{equation*}\label{c3}
|u|^\beta u = \Big(\Frac{(\beta+2)(l+1)}{d(l+2)}\Big)^{\frac{\beta+1}{\beta+2}}
r^{\frac{2(\beta+1)}{\beta+2}}\cos \theta|\cos \theta|^{\frac{\beta}{\beta+2}},
\end{equation*}
and
\begin{equation*}\label{c4}
c|u'|^\alpha u' = c r^{\frac{2(\alpha+1)}{l+2}}\sin \theta
|\sin \theta|^{\frac{2\alpha-l}{l+2}}.
\end{equation*}
After some elementary manipulations including a division by $\cos\theta$, we find that $\theta$ satisfies, at any non-singular point,  the differential equation
\begin{equation}\label{moT}\theta'+Ar^{\frac{2(\alpha-l)}{l+2}}\sin
\theta \cos \theta|\sin \theta|^{\frac{2(\alpha-l)}{l+2}}
+ B r^{\frac{2(\beta+1)}{\beta+2}-{\frac{2(l+1)}{l+2}}}|\cos
\theta|^{\frac{\beta}{\beta+2}}|\sin
\theta|^{\frac{-l}{l+2}} = 0 \end{equation}
with \begin{equation}\label{moT1}
A= c\frac{l+2}{2(l+1)}
;\quad B = \frac{l+2}{2(l+1)}\Big(\Frac{(\beta+2)(l+1)}{d(l+2)}\Big)^{\frac{\beta+1}{\beta+2}}
\end{equation}
On the other hand, we know that $r(t)$ tends to $0$ exactly
like $t^{\frac{-(l+2)}{2(\alpha-l)}}$
as $t$ tends to infinity and we suppose that $\alpha>l$.\\
In the case $\alpha > \Frac{\beta(l+1)+l}{\beta+2}$,\,\,\,\,\,
for $t$ large {\small\begin{equation}
 \begin{split}
\theta' \leq
-\eta t^{-\gamma} |\cos
\theta|^{\frac{\beta}{\beta+2}} |\sin
 \theta|^{\frac{-l}{l+2}},
\end{split}
\end{equation}}
where $\eta > 0$ and $\gamma =
\frac{l+2}{2(\alpha-l)}\Big(\frac{2(\beta+1)}{\beta+2}-\frac{2(l+1)}{l+2}\Big)<
\Frac{\beta-l}{\beta-l}=
1.$\\
In the case $\alpha =\Frac{\beta(l+1)+l}{\beta+2}$;\,\,\, $c
<c_0=
(\beta+2)\Big(\frac{(\beta+2)(l+1)}{d(\beta+1)(l+2)}\Big)^{\frac{\beta+1}{\beta+2}}$,
we obtain {\small\begin{equation*}
 \begin{split}
&\theta'=-c\frac{l+2}{2(l+1)}r^{\frac{2(\alpha-l)}{l+2}}\sin \theta
\cos \theta|\sin \theta|^{\frac{2(\alpha-l)}{l+2}} \\&-\frac{l+2}{2(l+1)}\Big(\Frac{(\beta+2)(l+1)}{d(l+2)}\Big)^{\frac{\beta+1}{\beta+2}}
r^{\frac{2(\beta+1)}{\beta+2}-{\frac{2(l+1)}{l+2}}}|\cos
\theta|^{\frac{\beta}{\beta+2}}|\sin
\theta|^{\frac{-l}{l+2}}
\\&\leq-
\frac{l+2}{2(l+1)}r^{\frac{2(\alpha-l)}{l+2}} |\cos
\theta|^{\frac{\beta}{\beta+2}}|\sin\theta|^{\frac{-l}{l+2}}\left\{c|\cos
\theta|^{-\frac{\beta}{\beta+2}}\sin \theta \cos \theta|\sin
\theta|^{\frac{2\alpha-l}{l+2}}+
\Big(\Frac{(\beta+2)(l+1)}{d(l+2)}\Big)^{\frac{\beta+1}{\beta+2}}\right\}.
\end{split}
\end{equation*}}
Since
$$
\max_{\theta \in R}(|\sin \theta|^{\frac{2(\alpha +1)}{l+2}}|\cos \theta|^{\frac{2}{\beta+2}})
= \Big(\frac{1}{\beta+2}\Big)^{\frac{1}{\beta+2}}\Big(\frac{\beta+1}{\beta+2}\Big)
^{\frac{\beta+1}{\beta+2}},
$$
the coefficient of $r^{\frac{2(\alpha-l)}{l+2}}$ is a positive
constant as soon as
$$
\Big(\frac{(\beta+2)(l+1)}{d(l+2)}\Big)^{\frac{\beta+1}{\beta+2}} -c
\Big(\frac{1}{\beta+2}\Big)^{\frac{1}{\beta+2}}\Big(\frac{\beta+1}{\beta+2}\Big)
^{\frac{\beta+1}{\beta+2}}>0,
$$
Therefore if $c < c_0$, then for $t$ large
{\small\begin{equation}
 \begin{split}\label{c5}
\theta' &\leq -\eta t^{-1}|\cos
\theta|^{\frac{\beta}{\beta+2}}|\sin \theta|^{\frac{-l}{l+2}},
 \end{split}
 \end{equation}} an inequality which is a fortiori satisfied when $\alpha > \Frac{\beta(l+1)+l}{\beta+2}$.
We introduce the function {\small\begin{equation*}
 \begin{split}
H(s) &= \Int_a^s \Frac{dv}{|\cos
v|^{\frac{\beta}{\beta+2}}|\sin v|^{\frac{-l}{l+2}}}\\&=
\Int_a^s \Frac{|\sin v|^{\frac{l}{l+2}}}{|\cos
v|^{\frac{\beta}{\beta+2}}}\,dv.
 \end{split}
 \end{equation*}}
Since u does not vanish for $t\geq T$, we may assume for instance
$$
\forall t\geq t_0,\,\,\,\,\,\, \theta(t)\in \Big(-\frac{\pi}{2},\frac{\pi}{2}\Big).
$$
Then, $H(\theta(t))=F(t)$ is differentiable for $t\geq t_0$,
using (\ref{c5}), we find
 {\small\begin{equation*}
 \begin{split}
 \forall t \geq t_0,\,\,\, F'(t) &\leq -\eta t^{-1}.
 \end{split}
\end{equation*}}
If we choose  $a=-\frac{\pi}{2}$ then $H(\theta(t))$ is non-negative,\,\,\, $H(\theta(t))$ is non-increasing. However the above inequality shows that $H(\theta(t))$ tends to $-\infty$ for $t $ large. This contradiction proves that $u'$ has a zero on each half-line. \\

Step 2. Applying Step 1 , we know that $u'$ has an infinite sequence of zeroes tending to infinity. We claim that between two successive zeroes of $u'$ there is a zero of $u$. Indeed let $u'(a)= u'(b) = 0$ with $a<b$ and $u' \not=0$ in $(a, b)$. If $u$ has a constant sign in $(a, b)$, by the equation $(|u'(t)|^{l}u'(t))' $ has the same sign for $t = a$ and $t= b$, which implies that $(|u'(t)|^{l}u'(t))' $ have opposite signs on $(a, a+\eta)$ and
$(b-\eta, b)$ for $\eta>0$ small enough, a contradiction with $u' \not=0$ in $(a, b)$. The proof of Theorem \ref{Th1} is now completed.\end{proof}
Our second result shows that when $\alpha$ is small, $u$ has a finite number of zeroes on each half-line.

\begin{theorem}\label{ml}
 Assume
 $$
 \alpha < \Frac{\beta(l+1)+l}{\beta+2}
 $$
 Then any solution $u(t)$ of (\ref{1}) which is not identically
 $0$ has a finite number of zeros on $(0,\infty)$. Moreover, for t
 large, $ u'(t)$ has the opposite sign to that of $u(t)$ and
 $u''(t) $ has the same sign as $u(t)$.
\end{theorem}
 \begin{proof}
 We introduce
$$
G(s) = \Int_0^s|\sin v|^{\frac{2\alpha+l}{l+2}}\sin v\cos v\,dv.
$$ First we observe that $G\circ\theta$ is $C^1$ on any interval where $u'$ does not vanish. Indeed on such an interval, $\theta$ is $C^1$ and multiplying (\ref{moT}) by $|\sin
\theta|^{\frac{2\alpha+l}{l+2}}\sin \theta \cos \theta$, we obtain
{\small\begin{equation*}
\begin{split}
[G(\theta(t))]'&=-Ar^{\frac{2(\alpha-l)}{l+2}}\sin^2
\theta \cos^2 \theta|\sin \theta|^{\frac{4\alpha-l}{l+2}}
\\&-B
r^{\frac{2(\beta+1)}{\beta+2}-{\frac{2(l+1)}{l+2}}}|\cos
\theta|^{\frac{\beta}{\beta+2}}|\sin
\theta|^{\frac{2\alpha}{l+2}}\sin \theta \cos \theta,
\end{split}
\end{equation*}} Then we observe that when $\sin\theta$ vanishes, the RHS of the above equality is 0. Actually it is also continuous at points where $\sin\theta$ vanishes, so that finally $G\circ\theta$ is $C^1$ everywhere. Now
using Cauchy-Schwarz inequality  we obtain {\small\begin{equation*}
\begin{split}
&-B
r^{\frac{2(\beta+1)}{\beta+2}-{\frac{2(l+1)}{l+2}}}|\cos
\theta|^{\frac{\beta}{\beta+2}}|\sin
\theta|^{\frac{2\alpha}{l+2}}\sin \theta \cos \theta\\&\leq B
r^{\frac{4(\beta+1)}{\beta+2}-{\frac{4(l+1)}{l+2}}}|\cos
\theta|^{\frac{\beta}{\beta+2}}|\sin
\theta|^{\frac{4\alpha}{l+2}}\sin^2 \theta \cos^2 \theta .
\end{split}
\end{equation*}}
Then
 {\small\begin{equation*}
\begin{split}
[G(\theta(t))]' & \leq
C(l,c,\beta,d)r^{\frac{2(\alpha-l)}{l+2}}\cos^2\theta \sin^2\theta
|\sin\theta|^{\frac{4\alpha-l}{l+2}}
\Big(1+r^{\frac{4(\beta+1)}{\beta+2}-\frac{4(l+1)}{l+2}-\frac{2(\alpha-l)}{l+2}}
|\sin \theta|^{\frac{l}{l+2}}|\cos\theta|^{\frac{\beta}{\beta+2}}\Big)\\&\leq
C(l,c,\beta,d)E(t)^{\frac{\alpha-l}{l+2}}\Big(r^{\frac{2(\alpha-l)}{l+2}}r^{-\frac{2(\alpha-l)}{l+2}}+
r^{\frac{4(\beta-l)}{(\beta+2)(l+2)}
-\frac{2(\alpha-l)}{l+2}}\Big)\\&\leq
C(l,c,\beta,d)E(0)^{\frac{\alpha-l}{l+2}}
r^{\frac{4(\beta-l)}{(\beta+2)(l+2)} -\frac{2(\alpha-l)}{l+2}}\leq
C'(\beta,c,l,d)t^{-\gamma},
\end{split}
\end{equation*}}
where $r(t)$ tends to exactly like
$t^{-\frac{(\alpha+1)(\beta+2)}{2(\beta-\alpha)}}$ and
{\small\begin{equation*}
\begin{split}
\gamma &=\Big(\frac{4(\beta-l)}{(\beta+2)(l+2)}
-\frac{2(\alpha-l)}{l+2}\Big)\frac{(\alpha+1)(\beta+2)}{2(\beta-\alpha)}
\\&=\frac{2(\beta-l)(\alpha+1)}{(l+2)(\beta-\alpha)}-
\frac{2(\alpha+1)(\alpha-l)}{(l+2)(\beta-\alpha)}-\frac{\beta(\alpha-l)
(\alpha+1)}{(l+2)(\beta-\alpha)}\\&=\frac{2(\alpha+1)}{l+2}
-\frac{\beta(\alpha-l)
(\alpha+1)}{(l+2)(\beta-\alpha)}\\&=\frac{2}{l+2}+
\frac{1}{l+2}\Big(2\alpha-
\frac{\beta(\alpha-l)\alpha+\beta(\alpha-l)-\alpha l+\alpha
l}{\beta-\alpha}\Big)\\&=\frac{2}{l+2}+
\frac{2\alpha(\beta-\alpha)-\beta(\alpha-l)\alpha-\beta\alpha+\alpha
l}{(\beta-\alpha)(l+2)}+\frac{\beta l-\alpha
l}{(\beta-\alpha)(l+2)}\\&=1+\alpha\Big[\frac{\beta(l+1)+l-\alpha(\beta+2)}
{(\beta-\alpha)(l+2)}\Big]>1,
\end{split}
\end{equation*}}
To finish the proof we shall use the following Lemma which is a straightforward extension of a lemma from {\cite{har2}} .
\begin{lemma}\label{L1}
Let $\theta \in C^0(a,+\infty)$ and $G$ be a non
constant T-periodic function. Assume that $G\circ\theta\in C^1(a,+\infty)$ and for some $h\in
L^1(a,+\infty)$
$$
\forall t\geq t_0,\,\,\,\,\, [G(\theta(t))]'\leq h(t).
$$
Then, for $t\geq t_1$ large enough, $\theta(t)$ remains in some
interval of length $\leq T$. If, in  addition, $G'$ has finite
number of zeroes on $[0,T]$. Then $\theta(t)$ has a limit for
$t\rightarrow \infty$.
\end{lemma}
\noindent {\bf{End of proof of Theorem \ref{ml}.}} By Lemma \ref{L1}, we know that there exists $\Theta $  such that $  \theta(t)\rightarrow \Theta $ as $t\mapsto +\infty. $
  If $\Theta\neq \frac{\pi}{2}$, then clearly $u$ has a constant sign for t large. If
 $\Theta=\frac{\pi}{2}$,\,\, $|u'|^{\frac{l}{2}+1}\sim r(t)$ and
$|u'|$ does not vanish for $t\geq A$. Then, $u$ can have at
most one zero $b$ in $(A,+\infty)$. In this case it has a constant
sign on $(b+1,+\infty)$. Next, let $t_0$ be such that $u$ has a
constant sign on $(t_0,+\infty)$. If $u'$ has several zeroes in
$(t_0,+\infty)$,  then  $(|u'|^lu')'$ must have different signs at
two successive  zeroes of $|u'|^lu'$, and by the equation the
corresponding values of $u$ must have different signs too, a
contradiction which shows that $u'$ has at most one zero in
$(t_0,+\infty)$ and therefore has a constant sign for t large. 
Since $ u(t)\rightarrow 0$ as $ t\mapsto +\infty,$ at infinity the signs of $u(t)$ and $u'(t)$ must
be opposite to each other. Finally it is easy to check that $u''(t)$ has the same sign as $u(t)$. Indeed, since $u'(t)$ tends to $0$ at infinity and keeps a constant sign, first of all $u''(t)$ cannot have the same sign as $u'(t)$ on any halfline $[T, \infty)$. Assuming for instance $u'>0$ on $[T, \infty)$ and $u''(a)<0$ with $a\ge T$, differentiating (\ref{1})yields
$$ (l+1) u''' = - \vert u'\vert ^{-l} (u''^2\vert u'\vert^{-1}  +(\alpha + 1)\vert u'\vert ^\alpha u'' + (\beta +1)\vert u\vert ^\beta u') \le-  (\alpha + 1)\vert u'\vert ^{\alpha-l} u''  $$ which clearly implies $u''< 0$ on $[a, \infty)$. The case $u'<0$ on $[T, \infty)$ follows by changing $u$ to $(-u)$.
This concludes the proof of Theorem  \ref{ml}\end{proof} We conclude this section by a result in the critical case
\begin{theorem}
Assume that
\begin{equation}
\alpha = \Frac{\beta(l+1)+l}{\beta+2};\,\,\,\,
c\geq c_0=(\beta+2)\Big(\frac{(\beta+2)(l+1)}{d(\beta+1)(l+2)}\Big)^{\frac{\beta+1}{\beta+2}}.
\end{equation}
Then any solution $u(t)$ of (\ref{1}) which is not identically
 $0$ has at most one zero on $(0,\infty)$.
\end{theorem}
\begin{proof}  In this case
 {\small\begin{equation*}
 \begin{split}
\theta'=-\frac{l+2}{2(l+1)}r^{\frac{2(\alpha-l)}{l+2}} |\sin
\theta|^{\frac{-l}{l+2}}\left\{\Big(\Frac{(\beta+2)(l+1)}{d(l+2)}\Big)^{\frac{\beta+1}{\beta+2}}
|\cos \theta|^{\frac{\beta}{\beta+2}}+c\sin \theta \cos
\theta|\sin \theta|^{\frac{2\alpha-l}{l+2}}\right\}.
\end{split}
\end{equation*}}
Introducing $$K(\theta)=\frac{l+2}{2(l+1)}|\sin
\theta|^{\frac{-l}{l+2}}\left\{\Big(\Frac{(\beta+2)(l+1)}{d(l+2)}\Big)^{\frac{\beta+1}{\beta+2}}
|\cos \theta|^{\frac{\beta}{\beta+2}}+c\sin \theta \cos
\theta|\sin \theta|^{\frac{2\alpha-l}{l+2}}\right\},$$
If $c= c_0$ then $K(\theta)> 0$, so that $\theta$ is non-increasing. Due to periodicity, the distance of two zeroes of $K(\theta)$ other than $\frac{\pi}{2}$ is not more than $\pi$ and therefore either $\theta$ remains in an interval of length less than $\pi$ or it coincides with one of these zeroes for a finite value of t. In the first case $\theta$, being non-increasing and bounded, converges to a limit and achieves at most once a value for which u vanishes. In the second case, due to existence and uniqueness for the ODE satisfied by $\theta$ near the non-trivial equilibria, $\theta$  is  constant and actually u never vanishes.\\
 If $c> c_0$,\,\,  $K(\theta)> 0$, If $\theta \neq \frac{\pi}{2}$, then it remains bounded and,
 since $K(\theta)> 0$ near the trivial zeroes, $\theta$ is monotone, and hence convergent. If $\theta= \frac{\pi}{2}$, then $\theta'=0$ and again u never vanishes. Otherwise, as previously, u vanishes at most once.\end{proof}
\section{A Detailed study of the non-oscillatory case}
\begin{theorem}\label{Non}
Assuming $l < \alpha < \frac{\beta(l+1)+l}{\beta+2}$,
any solution u of (\ref{1}) satisfies the following
 alternative: either there is a positive
constant $C$ such that
\begin{equation}\label{A}
\forall t \geq 1,\,\,\,\, E(t)\leq C(E(0))t^{-\frac{l+2}{\alpha-l}},
\end{equation}
or we  have
\begin{equation}\label{B}
\limsup_{t\rightarrow \infty}t^{\frac{(\alpha+ 1)(\beta+2)}
{\beta-\alpha)}}E(t)>0.
\end{equation}
\end{theorem}
\begin{proof}
As a consequence of Lemma \ref{L1}, we know that $\theta(t)$ tends to a limit $\Theta$ as $t\rightarrow \infty$. Moreover  if
$\sin \Theta \cos \Theta\neq 0$, we find as $t\rightarrow \infty$:
$$
\vert \theta'\vert  \geq  c\frac{l+2}{2(l+1)} r^{\frac{2(\alpha -l)}{l+2}}\vert \sin\Theta \cos\Theta\vert |\sin \Theta|^{\frac{2(\alpha-2)}{l+2}},
$$
and since by (\ref{22})  we have $r^{2\frac{\alpha-l}{l+2}}(t)\geq \frac{\eta}{t}$, this contradicts boundedness of $\theta(t)$.
we have only 2 possible cases\\
{\bf{Case 1. $\cos\Theta =0$}}. Then $|\sin\Theta|= 1$. In this case for t
large
$ |u'|^{l+2}\sim r^2(t),$
therefore
$$
|u'|^{\alpha +2}\sim r^{\frac{2(\alpha+2)}{l+2}}(t)=
 \Big(\frac{l+2}{l+1}E(t)\Big)^{\frac{\alpha+2}{l+2}}.
$$
Hence
$$
\frac{d}{dt}E(t) = -c|u'(t)|^{\alpha+2}\leq -\rho E^{\frac{\alpha+2}{l+2}}(t)
$$ for some $\rho>0$, then (\ref{A}) follows at once.\\
{\bf{Case 2.}} $\sin \Theta=0$. Then $|\cos \Theta|= 1$. In this case  as
$t\rightarrow \infty$
$$
r(t)\sim \Big(\frac{d(l+2)}{(\beta+2)(l+1)}\Big)^{\frac{1}{2}}|u|^{\frac{\beta}{2}+1},
$$
and
$$
\Frac{|u'|^{\frac{l}{2}}u'}{|u|^{\frac{\beta}{2}}u}=\frac{r(t)\sin \theta}{r(t)\cos\theta}\rightarrow 0.
$$
In particular for some $C>$ we have $ |u'(t)| \leq C|u|^{1+\frac{\beta-l}{l+ 2}} $ therefore for $t$ large enough and some $\gamma>0$, $|u(t)| \geq \gamma t^{-\frac{l+2}{\beta-l}}$. By
the non-oscillation result we may assume, replacing if necessary $u$ by $-u$,  that $u>0$ and $u'<0$ for
t large. Then, multiplying (\ref{moT}) by $|\sin \theta|^{\frac{l}{l+2}}$ and integrating it on $(t,2t)$, we obtain
 {\small\begin{equation*}
 \begin{split}
&0=\Int_t^{2t}\theta'|\sin\theta|^{\frac{l}{l+2}}\,ds+A\Int_t^{2t}r(s)^{\frac{2(\alpha-l)}{l+2}}\sin
\theta \cos \theta|\sin \theta|^{\frac{2\alpha-l}{l+2}} \,ds\\&+B
\Int_t^{2t}r(s)^{\frac{2(\beta+1)}{\beta+2}-\frac{2(l+1)}{l+2}}|\cos
\theta|^{\frac{\beta}{\beta+2}}\,ds\\&=
\Int_t^{2t}\theta'|\sin\theta|^{\frac{l}{l+2}}\,ds
+A\Int_t^{2t}r(s)^{\frac{2(\alpha-l)}{l+2}-1}r(s)\sin
\theta \cos \theta|\sin \theta|^{\frac{2\alpha-l}{l+2}} \,ds\\&+B
\Int_t^{2t}r(s)^{\frac{2(\beta+1)}{\beta+2}-\frac{2(l+1)}{l+2}}|\cos
\theta|^{\frac{\beta}{\beta+2}}\,ds \\&=
\Int_t^{2t}\theta'|\sin\theta|^{\frac{l}{l+2}}\,ds+A\Int_t^{2t}r(s)^{\frac{2(\alpha-l)}{l+2}-1}|u'|^{\frac{l}{2}}u' \cos \theta|\sin
\theta|^{\frac{2\alpha-l}{l+2}}\,ds
\\&+B
\Int_t^{2t}r(s)^{\frac{2(\beta+1)}{\beta+2}-\frac{2(l+1)}{l+2}}|\cos
\theta|^{\frac{\beta}{\beta+2}}\,ds.
\end{split}
\end{equation*}}Hence for t large
$$
\Int_t^{2t}r(s)^{\frac{2(\alpha-l)}{l+2}-1}|u'|^{\frac{l}{2}+1}
\cos \theta |\sin \theta|^{\frac{2\alpha-l}{l+2}}\,ds = -\Int_t^{2t}r(s)^{\frac{2(\alpha-l)}{l+2}-1}|u'|^{\frac{l}{2}}u'
\cos \theta |\sin \theta|^{\frac{2\alpha-l}{l+2}}\,ds $$ $$
 =\frac{1}{A}\Int_t^{2t}\theta'|\sin\theta|^{\frac{l}{l+2}}\,ds+
 \frac{B}{A}\Int_t^{2t}r(s)^{\frac{2(\beta-l)}{(\beta+2)(l+2)}}|\cos
\theta|^{\frac{\beta}{\beta+2}}\,ds $$ Introducing the function $$ F(x): = \int_0^x |\sin y|^{\frac{l}{l+2}}dy $$ we find $$\Int_t^{2t}\theta'|\sin\theta|^{\frac{l}{l+2}}\,ds = \Int_t^{2t}\theta'(s)F'(\theta(s))\,ds = F(\theta(2t))- F(\theta(t))$$ which tends to $0$ and in particular is bounded for $t$ large. Therefore, since $r(t)\sim
\Big(\frac{d(l+2)}{(\beta+2)(l+1)}\Big)^{\frac{1}{2}}|u|^{\frac{\beta}{2}}u$
and $u$ is positive, non-increasing, for t large since $\cos \theta$ approaches $1$ we deduce
\begin{equation}\label{int} \Int_t^{2t}r(s)^{\frac{2(\alpha-l)}{l+2}-1}|u'|^{\frac{l}{2}+1}
 |\sin \theta|^{\frac{2\alpha-l}{l+2}}\,ds \leq
 C_2\Big(1+t u(t)^{\frac{\beta-l}{l+2}}\Big)\leq 2C_2t u(t)^{\frac{\beta-l}{l+2}}.
\end{equation}  Now we observe that
$$ r^{\frac{2(\alpha-l)}{l+2}-1}|u'|^{\frac{l}{2}+1}
 |\sin \theta|^{\frac{2\alpha-l}{l+2}} = |u'|^{\frac{l}{2}+1}
 |r\sin \theta|^{\frac{2\alpha-l}{l+2}} r^{-1- {\frac{l}{l+2}}} = c_0  |u'|^{\frac{l}{2}+1+{\frac{2\alpha-l}{2}}} r^{-1- {\frac{l}{l+2}}}$$ for some positive constant $c_0$. After reduction this gives
 $$ r^{\frac{2(\alpha-l)}{l+2}-1}|u'|^{\frac{l}{2}+1}
 |\sin \theta|^{\frac{2\alpha-l}{l+2}} \sim  c_1  |u'|^{\alpha+1} u^{- {\frac{(l+1)(\beta+2)}{l+2}}}$$ and from \eqref{int}we now deduce
\begin{equation}
 \Int_t^{2t}|u'|^{\alpha+1} u^{- {\frac{(l+1)(\beta+2)}{l+2}}}ds\leq
 C_3tu(t)^{\frac{\beta-l}{l+2}},
 \end{equation}

By using Holder's inequality, we deduce
$$
 \Int_t^{2t}|u(s)|^{-\frac{(l+1)(\beta+2)}{(l+2)(\alpha+1)}}|u'(s)|\,ds \leq \Big(\Int_t^{2t}|u'|^{\alpha+1} u^{- {\frac{(l+1)(\beta+2)}{l+2}}}ds\Big)^{\frac{1}{\alpha+1}} t ^{\frac{\alpha}{\alpha+1}}\leq
 C_4tu(t)^{\frac{\beta-l}{(l+2)(\alpha+1)}}
$$
in other words
$$
u^{-\delta}(2t)-u^{-\delta}(t)=\Int_t^{2t} \frac{d}{ds}[u^{-\delta}(s)]\,ds\leq
 C_5tu(t)^{\frac{\beta-l}{(l+2)(\alpha+1)}}
$$
with $\delta= \frac{(l+1)(\beta+2)}{(l+2)(\alpha+1)}-1= \frac{(l+1)(\beta+2)-(l+2)(\alpha+1)}{(l+2)(\alpha+1)}$. First we check that $\delta>0$.  Indeed the condition on $\alpha$  implies $$\beta(l+1)+ l -\alpha(\beta+2)>0 $$
and on the other hand we have $$(l+1)(\beta+2)-(l+2)(\alpha+1)= (l+1)\beta+ 2l +2- (l+2)-(l+2)\alpha= (l+1)\beta+l-(l+2)\alpha$$ $ \ge \beta(l+1)+ l -\alpha(\beta+2)$ since $l\le \beta$. We claim that there is a set $S\subset (0,+\infty)$ containing arbitrarily large numbers such that for some $\nu> 0$
\begin{equation}\label{Z}
\forall t \in S,\,\,\, u^{-\delta}(2t)-u^{-\delta}(t)\geq \nu u^{-\delta}(t).
\end{equation}
Indeed by  Lemma 3.2 in {\bf\cite{{har2}}}, since $u$ decays as a negative power of $t$ we can find a set $S\subset (0,+\infty)$ containing arbitrarily large numbers for which
$$
\forall t\in S,\,\,\,\ u(2t)\leq e^{-\gamma}u(t).
$$
Hence
$$
\forall t\in S,\,\,\,\ u^{-\delta}(2t)\geq e^{\delta \gamma}u^{-\delta}(t),
$$
and then
$$
\forall t\in S,\,\,\,\ u^{-\delta}(2t)-u^{-\delta}(t)\geq (e^{\delta \gamma}-1)u^{-\delta}(t).
$$
Hence we have (\ref{Z}) with $\nu=e^{\delta \gamma}-1$. Now, we have for some $C_6> 0$ and $\sigma =\frac{1}{C_6}$
{\small\begin{equation*}
 \begin{split}
&\forall t \in S,\,\,\,
u^{-\delta}(t)\leq \frac{1}{\nu}(u^{-\delta}(2t)-u^{-\delta}(t))
\leq C_6tu(t)^{\frac{\beta-l}{(l+2)(\alpha+1)}}
\\&\Rightarrow \forall t \in S,\,\,\, u^{\delta+{\frac{\beta-l}{(l+2)(\alpha+1)}}}(t)\geq \sigma t^{-1},
\end{split}
\end{equation*}}
with
{\small\begin{equation*}
 \begin{split}
\delta+{\frac{\beta-l}{(l+2)(\alpha+1)}}&=\frac{(l+1)(\beta+2)-(l+2)(\alpha+1)+\beta-l }{(l+2)(\alpha+1)}
\\&= \frac{l(\beta+1)-(\alpha+1)l+2(\beta-\alpha) }{(l+2)(\alpha+1)}= \frac{\beta-\alpha }{\alpha+1}
\end{split}
\end{equation*}}
finally, we find for some  $\sigma' > 0$
$$
\forall t \in S,\,\,\,\, u(t)\geq \sigma' t^{-\frac{\alpha+1}
{\beta-\alpha}}.
$$
We see from (\ref{2}) that
$$
\forall t \in S,\,\,\,\, E(t)\geq \frac{1}{\beta+2} |u(t)|^{\beta+2} \geq \sigma'' t^{-\frac{(\alpha+1)(\beta+2)}
{\beta-\alpha}},
$$
therefore
$$
\limsup_{t\rightarrow \infty}t^{\frac{(\alpha+1)(\beta+2)}
{\beta-\alpha}}E(t)>0.
$$
The proof of Theorem \ref{Non} is now completed.\end{proof}
\begin{remark}
Assume  $l < \alpha < \frac{\beta(l+1)+l}{\beta+2}$. Then if a  solution u of (\ref{1}) satisfies (\ref{A}) there exists a positive constant $C$ such that
$$
\forall t \geq 1,\,\,\, |u(t)|\leq Ct^{-\frac{l+1-\alpha}{\alpha-l}},
$$
and
$$
\forall t \geq 1,\,\,\, |u'(t)|\leq Ct^{-\frac{1}{\alpha-l}}.
$$
\end{remark}

\begin{remark} The next result shows that there are actually some solutions u of (\ref{1}) satisfiying (\ref{A})
\end{remark}
\begin{theorem}\label{C}
Let $l < \alpha < \frac{\beta(l+1)+l}{\beta+2}$. Then, there exists a solution $u> 0$ of
(\ref{1}) such that for some constant $C> 0$
$$
\forall t \geq 0,\,\,\, u(t)\leq C(1+t)^{-\frac{1-\alpha+l}{\alpha-l}},\,\,\,\,\,\,
|u'(t)|\leq C(1+t)^{-\frac{1}{\alpha-l}}.
$$
\end{theorem}
\begin{proof} By homogeneity it is sufficient to prove the result for $c=d=1.$ We introduce two Banach spaces $X$ and $Y$ as follows
$$
X=\{z\in C([1,+\infty),\,\,\, t^{\frac{l+1}{\alpha-l}}z(t)
\in L^{\infty}[1,+\infty))\},
$$
with norm
$$
\forall z \in X,\,\,\,\,\, \|z\|_X =
\|t^{\frac{l+1}{\alpha-l}}z(t)\|_{L^\infty([1,+\infty))},
$$
and
$$
Y=\{z\in C[1,+\infty),\,\,\,\, t^{\frac{\alpha+1}{\alpha-l}}z(t) \in L^\infty([1,+\infty))\},
$$
$$
\forall z \in Y,\,\,\,\, \|z\|_Y= \|t^{\frac{\alpha+1}{\alpha-l}}z(t)\|_{L^\infty([1,+\infty))}.
$$ for convenience we also consider 
$$ X^+ = 
\{z\in X, z\ge 0 \} 
$$ and 
$$ Y^+ = 
\{z\in Y, z\ge 0 \} 
$$ 
The proof proceeds in $3$ steps.\\\\
 {\bf{Step 1. A Preliminary Estimate.}}  Let $f \in Y^+,\,\,\,\varphi
 \in \R^+$ and consider the problem
 \begin{equation}\label{33}
 v'+v^{\frac{\alpha+1}{l+1}}=f;\,\,\,\,\, v(1)=\varphi.
 \end{equation}
 \begin{lemma}\label{W}
 Under the conditions
 $$
 |\varphi|\leq \Big(\frac{l+2}{\alpha-l}\Big)^{\frac{l+1}{\alpha-l}},\,\,\,\,
 \|f\|_Y\leq \Big(\frac{l+2}{\alpha-l}\Big)^{\frac{l+1}{\alpha-l}}
 \Big(\frac{1}{\alpha-l}\Big)
 $$
 the unique solution $v$ of (\ref{33}) is in $X^+$ with
 $$
 \|v\|_X\leq \Big(\frac{l+2}{\alpha-l}\Big)^{\frac{l+1}{\alpha-l}}.
 $$
 \end{lemma}
 {\bf{Proof.}}
  Let
 \begin{equation}\label{HH}
 w(t)= \Big(\frac{l+2}{\alpha-l}\Big)^{\frac{l+1}{\alpha-l}}t^{-\frac{l+1}{\alpha-l}}.
 \end{equation}
 We have
 $w(1)=\Big(\frac{l+2}{\alpha-l}\Big)^{\frac{l+1}{\alpha-l}}$ and
 {\small\begin{equation*}
 \begin{split}
 w'+w^{\frac{\alpha+1}{l+1}}&=
 \Big(\frac{l+2}{\alpha-l}\Big)^{\frac{1}{\alpha-l}}
 \Big[-\Big(\frac{l+2}{\alpha-l}\Big)^{\frac{l}{\alpha-l}}
 \Big(\frac{l+1}{\alpha-l}\Big)+\Big(\frac{l+2}{\alpha-l}\Big)^{\frac{\alpha}{\alpha-l}}\Big]
 t^{-\frac{\alpha+1}{\alpha-l}}\\&=
 \Big(\frac{l+2}{\alpha-l}\Big)^{\frac{l+1}{\alpha-l}}
  \Big(\frac{1}{\alpha-l}\Big)
 t^{-\frac{\alpha+1}{\alpha-l}}\ge f.
 \end{split}
\end{equation*}} Then the result follows from the elementary comparison principle for first order ODE with a locally Lipschitz continuous nonlinearity. \\

{\bf{Step 2. An Integrodifferential problem. }}
We introduce the integral operator $\mathcal{K}$ defined on the positive cone $\mathcal{C}$ of $L^{1}_{loc}([1,+\infty))$ by
$$
\forall v \in \mathcal{C},\,\,\,\, \forall t \in [1,+\infty),\,\,\,\,
\mathcal{K}(v)(t)=\Big|\Int_t^{\infty}v^{\frac{1}{l+1}}(s)\,ds\Big|^\beta
\Int_t^\infty v^{\frac{1}{l+1}}(s)\,ds.
$$
We claim that for each $v \in X^+$ , 
$\mathcal{K}(v)$ is finite  everywhere and moreover  $\mathcal{K}(X^+)\subset Y^+$ with
\begin{equation}\label{X}
\forall v\in X^+,\,\, \|\mathcal{K}(v)\|_Y\leq C\|v\|_X^{\beta+1}.
\end{equation}
Indeed, since clearly $\alpha < 1+l$, an easy calculation shows that  for all $v\in X^+$, $v^{\frac{1}{l+1}}$ is integrable on $([1,+\infty))$. Moreover, for all $t\in [1,+\infty)$ we have

{\small\begin{equation*}
 \begin{split}
|\mathcal{K}(v)(t)|&=
\Big|\Int_t^{\infty}v^{\frac{1}{l+1}}(s)\,ds\Big|^{\beta+1}\\&=
\Big|\Int_t^{\infty}s^{\frac{1}{\alpha-l}}v^{\frac{1}{l+1}}(s)s^{-\frac{1}{\alpha-l}}
\,ds\Big|^{\beta+1}\\&=
\Big|\Int_t^{\infty}\Big(s^{\frac{l+1}{\alpha-l}}v(s)\Big)^{\frac{1}{l+1}}s^{-\frac{1}{\alpha-l}}
\,ds\Big|^{\beta+1}
\\&\leq \Big| \Big(\sup_{t\geq 1} t^{\frac{l+1}{\alpha-l}}v(t)\Big)^{\frac{1}{l+1}}\Int_t^\infty
s^{-\frac{1}{\alpha-l}}\,ds\Big|^{\beta+1}\\& \leq
\Big(C(\alpha,l)\|v(t)\|_X^{\frac{1}{l+1}}
t^{1-\frac{1}{\alpha-l}}\Big)^{\beta+1}\\&
 \leq
C'(\alpha,\beta,l)\|v(t)\|^{\frac{\beta+1}{l+1}}_X
t^{(-\frac{1}{\alpha-l}+1)(\beta+1)}.
\end{split}
\end{equation*}}
But
$$
\alpha < \frac{\beta(l+1)+l}{\beta+2}\Rightarrow \alpha-l <
\frac{\beta-l}{\beta+2}
$$
{\small\begin{equation*}
 \begin{split}
 \frac{\beta-l}{\alpha-l}-(\beta+1)-1> 0
 &\Rightarrow
 \frac{\beta-l}{\alpha-l}-(\beta+1)-1+\frac{\alpha+1}{\alpha-l}>\frac{\alpha+1}{\alpha-l}\\&\Rightarrow
 \frac{\beta-l}{\alpha-l}
-(\beta+1)+\frac{l+1}{\alpha-l}>\frac{\alpha+1}{\alpha-l}
\\&\Rightarrow (\beta+1)\Big(\frac{1}{\alpha-l}-1\Big)>
\frac{\alpha+1}{\alpha-l},
\end{split}
\end{equation*}} and (\ref{X}) follows easily. Now,
we consider for $\varepsilon$ small enough the solution $z= \mathcal{T}(v)$
of the perturbed problem
$$ z'+z^{\frac{\alpha+1}{l+1}}=\varepsilon \mathcal{K} v,\,\,\,\, z(1)=\varphi $$
Let
$$ B:=\Big\{z\in X^+,\,\,\,\,\ \|z\|_X \leq \Big(\frac{l+2}{\alpha-l}\Big)
^{\frac{l+1}{\alpha-l}}\Big\} $$ and fix $\varepsilon>0$ small enough to insure
\begin{equation}
\varepsilon C\Big(\frac{l+2}{\alpha-l}\Big)^{\frac{\beta+1}{\alpha-l}}\leq
\Big(\frac{1}{\alpha-l}\Big)\Big(\frac{l+2}{\alpha-l}\Big)^{\frac{l+1}{\alpha-l}}
\end{equation}
As a consequence of Lemma \ref{W}, we have $\mathcal{T}(B) \subset B$.\\

{\bf{Step 3. An Iterative Scheme. }} We consider the sequence $v_n=\mathcal{T}^n(0)$ defined as follows: $v_1$ is the solution of
$$ v'+v^{\frac{\alpha+1}{l+1}} =0,\,\,\,\, v(1)=\varphi. $$
Clearly, $v_1$ is non-negative, non-increasing and $v_1 \in B$. When $v_n$ is know, we define $v_{n+1}$ as the solution of
$$ v'+v^{\frac{\alpha+1}{l+1}}=\varepsilon\mathcal{K}v_n ,\,\,\,\, v(1)=\varphi.$$
Then $v_n$ is non-increasing, non-negative and bounded by a fixed positive element of $X$, ${v_n}^{\frac{1}{l+1}}$ is bounded by a fixed function of $L^1([1,+\infty))$  and since $v'_n$ is uniformly bounded, $v_n$ converges locally uniformly and $v_n^{\frac{1}{l+1}}$ converges in
$L^1([1,+\infty))$. The limit $v$ of $v_n$ is a solution of
$$
v'+v^{\frac{\alpha+1}{l+1}}=\varepsilon\mathcal{K}v ,\,\,\,\, v(1)=\varphi.
$$
{\bf{Step 4. Conclusion. }}
Therefore, $v$ is a positive solution of
\begin{equation}\label{000}
v'+v^{\frac{\alpha+1}{l+1}}= \varepsilon \Big|\Int_t^\infty v^{\frac{1}{l+1}}(s)\,ds\Big|^\beta
\Int_t^\infty v^{\frac{1}{l+1}}(s)\,ds,\,\,\,\, v(1)=\varphi.
\end{equation}
let
\begin{equation}\label{S}
\forall t \geq 0,\,\,\,\ u(t)=\Int_{t+1}^\infty v^{\frac{1}{l+1}}(s)\,ds.
\end{equation}
Then $u\geq 0$ and
$$u'=-v^{\frac{1}{l+1}}(.+1),\,\,\,\,\,\, (|u'|^{l}u')'=
-\Big(|v(.+1)|^{\frac{l}{l+1}}v^{\frac{1}{l+1}}(.+1)\Big)'=-v'.$$
Hence, (\ref{000}) rewrites as
$$
-(|u'|^l u')'-|u'|^\alpha u' = \varepsilon |u|^\beta u.
$$
Since $u\geq 0$, we get
$$
v'+v^{\frac{\alpha+1}{l+1}}\leq 0
\Rightarrow \frac{v'}{v^{\frac{\alpha+1}{l+1}}}\leq -1,
$$
integrating over $(1,t)$, we have
$$
-\frac{l+1}{\alpha-l}v^{-\frac{\alpha-l}{l+1}}(t)\leq -t+1-\frac{l+1}{\alpha-l}v^{-\frac{\alpha-l}{l+1}}(1),
$$
therefore
$$
v^{-\frac{\alpha-l}{l+1}}(t)\geq C(1+t),
$$
where $\alpha> l$, we get
$$
|v(t)|\leq C(1+t)^{-\frac{l+1}{\alpha-l}}.
$$
Since $v\in X$ and using (\ref{S}), we have finally
$$
|u(t)| \leq C_1(1+t)^{-(\frac{1}{\alpha-l}-1)},\,\,\,\,
 |u'(t)|\leq C_2(1+t)^{-\frac{1}{\alpha-l}}.
$$
This concludes the proof of Theorem \ref{C}.\end{proof}
\begin{theorem}
Let $\alpha < \frac{\beta(l+1)+l}{\beta+2},\,\,c>0,\,\,d>0$. Then (\ref{1}) has an open set of initial data leading to a  slow solution, which means a solution satifying \eqref{B}.
\end{theorem}
\begin{proof}
For any solution $u$ of the equation we introduce the new coordinates $(z, w)$ defined by 
$$
z=\sqrt{\frac{d(l+2)}{(\beta+2)(l+1)}}|u|^{\frac{\beta}{2}}u,\,\,\, w=|u'|^{\frac{l}{2}}u',
$$ In particular we have 
$$
u'=|w|^{\frac{2}{l+2}}\sgn(w);\quad 
|u|^{\frac{\beta}{2}}= \Big(\frac{(\beta+2)(l+1)}{d(l+2)}\Big)^{\frac{\beta}{2(\beta+2)}}|z|^{\frac{\beta}{\beta+2}}
$$
and since
\begin{equation*}
\begin{split}
z'&=\frac{\beta+2}{2}\sqrt{\frac{d(l+2)}{(\beta+2)(l+1)}} |u|^{\frac{\beta}{2}}u'\\&
=\frac{\beta+2}{2}\Big(\frac{(\beta+2)(l+1)}{d(l+2)}\Big)^{-\frac{1}{2}} 
 \Big(\frac{(\beta+2)(l+1)}{d(l+2)}\Big)^{\frac{\beta}{2(\beta+2)}}|z|^{\frac{\beta}{\beta+2}}|w|^{\frac{2}{l+2}}\sgn(w)\\&=
 \frac{\beta+2}{2}\Big(\frac{(\beta+2)(l+1)}{d(l+2)}\Big)^{-\frac{1}{\beta+2}} |z|^{\frac{\beta}{\beta+2}}
 |w|^{\frac{2}{l+2}}\sgn(w)
 \\&=\frac{d(l+2)}{2(l+1)}\Big(\frac{(\beta+2)(l+1)}{d(l+2)}\Big)^{\frac{\beta+1}{\beta+2}} |z|^{\frac{\beta}{\beta+2}}
 |w|^{\frac{2}{l+2}}\sgn(w)
\end{split}
\end{equation*}
we find the equation for $z$: 
\begin{equation}\label{E1}
z'=a |z|^{\frac{\beta}{\beta+2}}
 w^{\frac{2}{l+2}},
\end{equation}
with $a=\frac{d(l+2)}{2(l+1)}\Big(\frac{(\beta+2)(l+1)}{d(l+2)}\Big)^{\frac{\beta+1}{\beta+2}} > 0$

Similarly we have 
$$
|u'|^\alpha u'= |w|^{\frac{2\alpha+2}{l+2}}\sgn(w)
$$
$$
|u|^\beta u= \Big(\frac{(\beta+2)(l+1)}{d(l+2)}\Big)^{\frac{\beta+1}{\beta+2}}|z|^{\frac{\beta}{\beta+2}} z,
$$
\begin{equation*}\Big(|u'|^lu'\Big)'=
(|w|^{\frac{l}{l+2}}w)'= 
\frac{2(l+1)}{l+2}|w|^{\frac{l}{l+2}}w'
\end{equation*} Hence by the equation 
\begin{equation*}
\begin{split}
\frac{2(l+1)}{l+2}|w|^{\frac{l}{l+2}}w'= -d\Big(\frac{(\beta+2)(l+1)}{d(l+2)}\Big)^{\frac{\beta+1}{\beta+2}}|z|^{\frac{\beta}{\beta+2}} z
-c|w|^{\frac{2\alpha+2}{l+2}}\sgn(w)
\end{split}
\end{equation*} which gives the equation in $w$
$$
w'=-\frac{d(l+2)}{2(l+1)}\Big(\frac{(\beta+2)(l+1)}{d(l+2)}\Big)^{\frac{\beta+1}{\beta+2}}|w|^{-\frac{l}{l+2}}|z|^{\frac{\beta}{\beta+2}} z-c\frac{l+2}{2(l+1)}
|w|^{\frac{2\alpha-l+2}{l+2}}\sgn(w)
$$
\begin{equation}\label{E2}
w'=-a|w|^{-\frac{l}{l+2}}|z|^{\frac{\beta}{\beta+2}} z-c\frac{l+2}{2(l+1)}
|w|^{\frac{2\alpha-l+2}{l+2}}\sgn(w)
\end{equation} valid whenever $w\not=0 $. 
For $u<0,\,\,u'>0$, we consider the region $S_{\varepsilon,M}$
$$
S_{\varepsilon,M}=\left\{(z,w)\in \R^2/ z< 0,\,\, z^2+ w^2< \varepsilon^2,\,\, 0<\frac{w}{|z|}<M\right\}.
$$
 For any finite $M$ given in advance, we shall show that for $\varepsilon$ small enough, the region $S_{\varepsilon,M}$ is positively invariant. To this end we introduce the vector
 $$ F(z,w) : = \Big(a |z|^{\frac{\beta}{\beta+2}}
 w^{\frac{2}{l+2}},-a|w|^{-\frac{l}{l+2}}|z|^{\frac{\beta}{\beta+2}} z-c\frac{l+2}{2(l+1)}
|w|^{\frac{2\alpha-l+2}{l+2}}\sgn(w)\Big)$$ so that as long $(z,w)$ remains in $
S_{\varepsilon,M}$ we have the equation
$$(z',w')= F(z,w) $$
 Setting
 $B_\varepsilon =\{(z,w)\in \R^2/ z^2+w^2 \leq \varepsilon^2\}$ , 
 since
 $$ \langle F(z,w),(z,w)\rangle=-c\frac{l+2}{2(l+1)}|w|^{\frac{2(\alpha+2)}{l+2}}\leq 0, 
 $$
 we find that the solution cannot escape  $S_{\varepsilon,M}$ at a point of  $\partial B_\varepsilon .$\\
 
   By backward  uniqueness it is clear  that 
   $(z,w)$  cannot leave $ S_{\varepsilon,M} $ through $(0, 0). $ We now show that if $\varepsilon$ is small enough, the solution cannot escape at any point of 
  $$ \triangle_M=\{(-\lambda,M\lambda),\,\, \lambda\in (0,+\infty)\}
  $$ lying in the closure of  $B_\varepsilon $.
  Indeed we have
  $$
  F(-\lambda,M\lambda)=
  \Big(aM^{\frac{2}{l+2}}\lambda^{\frac{\beta}{\beta+2}+\frac{2}{l+2}},
  aM^{\frac{-l}{l+2}}\lambda^{\frac{\beta}{\beta+2}+\frac{2}{l+2}}-
  cM^{\frac{2\alpha-l+2}{l+2}}\frac{l+2}{2(l+1)}\lambda^{\frac{2\alpha-l}{l+2}+\frac{2}{l+2}}\Big)
  $$
  Since $\frac{2\alpha-l}{l+2}<\frac{\beta}{\beta+2}$ as a consequence of $
  \alpha < \frac{\beta(l+1)+l}{\beta+2}$, for $\lambda$ small enough the field at any point of $
  \triangle_M $ points into the region $ S_{\varepsilon,M} $. And smallness of $\lambda$ is a consequence of smallness of $\varepsilon$ whenever $(z,w)\in \triangle_M. $  
  
  Finally, since $ F(-\lambda,w)$  tends to $(0,+\infty)$ as $w\rightarrow 0$ , the solution cannot escape 
  $ S_{\varepsilon,M} $ at a point lying on the horizontal axis. 
  More precisely, assuming the contrary means  that for some finite $t_0>0$ we have $w>0$ on $[t_0-\delta,t_0)$,\,\,$w(t_0)=0,$\,\,  and $z(t_0)<0. $ Then for t sufficiently close to $t_0$:
  \begin{equation*}
\begin{split}
  w'&\geq C_1 w^{-\frac{1}{l+2}}-C_2w^{\frac{2\alpha_l+2}{l+2}}\\&
  \geq C w^{-\frac{1}{l+2}}
 \end{split}
  \end{equation*}
  so that  $w(t)$ is increasing for t sufficiently close to $t_0$, and this contradicts $w(t_0)=0$. Finally, for any trajectory of (\ref{E1}) and (\ref{E2}) lying in any region $S_{\varepsilon,M}$, $\frac{w}{|z|}=|\tan \theta| $ is bounded , when  for a fast solution $ |\tan \theta| $ blows-up at infinity in $t$. Hence all solutions confined in $S_{\varepsilon,M}$ are slow solutions. \end{proof}

\section{Optimality Results}
\begin{proposition}
Assume either (\ref{CO1}) or (\ref{CO2}). Then, the results of
Corollary \ref{CO} are optimal. More precisely, any solution $u\not\equiv 0$ of (\ref{1}) satisfies
\begin{equation}\label{RT}
\limsup_{t\rightarrow +\infty}t^{\frac{l+2}{(\beta+2)(\alpha-l)}}u(t)> 0
\end{equation}
\begin{equation}\label{RTT}
\limsup_{t\rightarrow +\infty}t^{\frac{1}{\alpha-l}}u'(t)> 0.
\end{equation}
\end{proposition}
\begin{proof} As a consequence of Theorem \ref{Th1}, there is a sequence $t_n\rightarrow +\infty$ such that
$$
u(t_n)> 0,\,\,\, u'(t_n)=0.
$$
Using (\ref{2}), we obtain
$$
u(t_n)=\left\{\frac{\beta+2}{d}E(t_n)\right\}^{\frac{1}{\beta+2}},
$$
and (\ref{RT}) is a consequence of  (\ref{22}). Similarly, there is a sequence $\tau_n\rightarrow +\infty$ such that
$$
u(\tau_n)= 0,\,\,\, u'(\tau_n)>0.
$$
Using (\ref{2}), we obtain
$$
u'(\tau_n)=\left\{\frac{l+2}{l+1}E(\tau_n)\right\}^{\frac{1}{l+2}},
$$
and (\ref{RTT}) is a consequence of  (\ref{22}).\end{proof}
\begin{theorem}\label{mp}
Assume  $\alpha < \frac{\beta(l+1)+l}{\beta+2}$ and $\alpha>l$. Then for any  solution $u\neq 0$ of (\ref{1}) satisfying (\ref{A}), we have
$$
\lim_{t\rightarrow +\infty}t^{\frac{1-\alpha+l}{\alpha-l}}|u(t)|=\frac{\alpha-l}{1-\alpha+l}
\Big(\frac{l+1}{c(\alpha-l)}\Big)^{\frac{1}{\alpha-l}}.
$$
$$
\lim_{t\rightarrow +\infty}t^{\frac{1}{\alpha-l}}|u'(t)|=
\Big(\frac{l+1}{c(\alpha-l)}\Big)^{\frac{1}{\alpha-l}}.
$$
\end{theorem}
\begin{proof} We consider the case $u> 0,\,\,\,u'< 0$. Then, equation (\ref{1}) implies
\begin{equation}\label{Q}
(|u'|^lu')'+c(-u')^{\alpha}u'=-du^{\beta+1} < 0,
\end{equation}
so that if $v^{\frac{1}{l+1}}=-u'$ we find
\begin{equation}\label{V1}
v'\geq -cv^{\frac{1+\alpha}{l+1}},
\end{equation}
for t large enough, using the fact that $\cos \theta(t)$ tends to $0$, we have
$$
|u(t)|\leq |u'(t)|^{\frac{l+2}{\beta+2}},
$$
hence,
\begin{equation}\label{QQ}
|u(t)|^{\beta+1}\leq |u'(t)|^{\frac{(\beta+1)(l+2)}{\beta+2}}.
\end{equation}
Using (\ref{Q}) and (\ref{QQ}), we find
$$ v'+cv^{\frac{\alpha+1}{l+1}}\leq
|u'|^{\frac{(l+2)(\beta+1)}{\beta+2}}=v^{\frac{(l+2)(\beta+1)}{(\beta+2)(l+1)}} $$ 

Since

$\frac{(l+2)(\beta+1)}{\beta+2}-(\alpha+1)=\frac{\beta(l+1)+l}{\beta+2}-\alpha>
0$, we find
\begin{equation}\label{V2}
v'\leq (-c+\varepsilon(t))v^{\frac{\alpha+1}{l+1}}
\end{equation}
where 
 $\lim_{t\rightarrow +\infty}\varepsilon(t)=0$. It follows easily that $$\lim_{t\rightarrow +\infty}t^{\frac{1}{\alpha-l}} |u'(t)|
=  \lim_{t\rightarrow +\infty}t^{\frac{1}{\alpha-l}} v(t)^{\frac{1}{l+1}}= \Big(\frac{l+1}{c(\alpha-l)}\Big)^{\frac{1}{\alpha-l}}.
$$
Finally we have
$$\forall t \geq 1,\quad u(t)=\Int_t^\infty v^{\frac{1}{l+1}}(s)\,ds $$ and this is easily seen to imply
$$ \lim_{t\rightarrow+\infty}t^{\frac{1-\alpha+l}{\alpha-l}}|u(t)|=\frac{\alpha-l}{\alpha-1-l}
\Big(\frac{l+1}{c(\alpha-l)}\Big)^{\frac{1}{\alpha-l}}.
$$
This ends the proof of Theorem \ref{mp}. \end{proof}

\end{document}